\documentclass[11pt,reqno]{amsart}
\usepackage{amsmath,amssymb, amscd, amsmath, color, enumitem}
\usepackage{mathrsfs, bm, cite}
\usepackage{graphics}
\usepackage{tikz}
\usepackage{pgfplots}
\usepackage{pgfplotstable}
\usepackage{tkz-fct}
\usetikzlibrary{pgfplots.polar}
\pgfplotsset{compat=newest}
\usepgfplotslibrary{fillbetween}
\usetikzlibrary{patterns}
\input epsf

\numberwithin{equation}{section}
\newtheorem{theorem}{Theorem}[section]
\newtheorem{lemma}[theorem]{Lemma}
\newtheorem{corollary}[theorem]{Corollary}

\theoremstyle{definition}

\newtheorem{example}[theorem]{Example}
\theoremstyle{remark}
\newtheorem{remark}[theorem]{Remark}

\def\XXint#1#2#3{{\setbox0=\hbox{$#1{#2#3}{\int}$ }
\vcenter{\hbox{$#2#3$ }}\kern-.6\wd0}}

\begin{document}

\title[A Unified Topological Analysis of Variable Growth Kirchhoff-Type Equations]{A Unified Topological Analysis of Variable Growth Kirchhoff-Type Equations}

\author[C. S. Goodrich]{Christopher S. Goodrich}
\address{School of Mathematics and Statistics\\
UNSW Sydney\\
Sydney, NSW 2052 Australia}
\email[Christopher S. Goodrich]{c.goodrich@unsw.edu.au}
\author[G. Nakhl]{Gabriel Nakhl}
\address{School of Mathematics and Statistics\\
UNSW Sydney\\
Sydney, NSW 2052 Australia}
\email[Gabriel Nakhl]{g.nakhl@student.unsw.edu.au}
\keywords{Nonlocal differential equation; Kirchhoff-type equation; variable growth; positive solution; convolution.}
\subjclass[2010]{Primary: 33B15, 34B10, 34B18, 42A85, 44A35.  Secondary: 26A33, 46E35, 47H30.}


\begin{abstract}

We consider the nonlocal differential equation
\begin{equation}
-A\left(\int_0^1b(1-s)\big(u(s)\big)^{p(s)}\ ds\right)u''(t)=\lambda f\big(t,u(t)\big)\text{, }t\in(0,1),\notag 
\end{equation}
which is a one-dimensional Kirchhoff-like equation with a nonlocal convolution coefficient.  The novelty of our work involves allowing a variable growth term in the nonlocal coefficient.  By relating the variable growth problem to a constant growth problem, we are able to deduce existence of at least one positive solution to the differential equation when equipped with boundary data.  Our methodology relies on topological fixed point theory.  Because our results treat both the convex and concave regimes, together with both the variable growth and constant growth regimes, our results provide a unified framework for one-dimensional Kirchhoff-type problems.

\end{abstract}

\maketitle

\section{Introduction}

In this paper we consider the existence of at least one positive solution to the nonlocal differential equation
\begin{equation}\label{eq1.1}
-A\left(\left(b*u^{p(\cdot)}\right)(1)\right)u''(t)=\lambda f\big(t,u(t)\big)\text{, }t\in(0,1)
\end{equation}
subject to given boundary data, which can assume a variety of forms such as Dirichlet $u(0)=0=u(1)$ or right-focal $u(0)=0=u'(0)$ to name but two possibilites.  In \eqref{eq1.1} we note that $A\ : \ [0,+\infty)\rightarrow\mathbb{R}$, $f\ : \ [0,1]\times[0,+\infty)\rightarrow[0,+\infty)$, and $p\ : \ [0,1]\rightarrow(0,+\infty)$ are continuous functions, and $b$ is an almost everywhere positive $L^1$ function.  Note that the finite convolution in \eqref{eq1.1} is interpreted as
\begin{equation}
\left(b*u^{p(\cdot)}\right)(1):=\int_0^1b(1-s)\big(u(s)\big)^{p(s)}\ ds.\notag
\end{equation}
The function $p$, in particular, is assumed to satisfy
\begin{equation}
0<p^-<p(t)<p^+<+\infty,\notag
\end{equation}
where, depending upon the section, the constants $p^-$ and $p^+$ live in specified intervals -- e.g., $1<p^-<p^+$.  Our results in this paper treat all three potential regimes.
\begin{enumerate}
\item $1<p^-<p^+$

\item $0<p^-<p^+\le1$

\item $p^-<1<p^+$
\end{enumerate}
In this sense, the results provide a unified framework for studying problem \eqref{eq1.1}.

Regarding \eqref{eq1.1}, note that if $b(t)\equiv1$ and $p(t)\equiv p_0>0$, then we are lead to the important model case
\begin{equation}
-A\left(\Vert u\Vert_{L^{p_0}}^{p_0}\right)u''(t)=\lambda f\big(t,u(t)\big),\notag
\end{equation}
which is closely related to a one-dimensional steady-state version of the classical Kirchhoff parabolic PDE
\begin{equation}
u_{tt}-A\left(\Vert Du\Vert_{L^2}^{2}\right)\Delta u=\lambda f\big(\bm{x},u(\bm{x})\big).\notag
\end{equation}
We note that the inclusion of the kernel $b$ in the convolution in \eqref{eq1.1} allows us to incorporate a variety of physically meaningful nonlocal elements, chief amongst these being if, for $t>0$, we put $\displaystyle b(t):=\frac{1}{\Gamma(\alpha)}t^{\alpha-1}$, $0<\alpha<1$, which leads to a Riemann-Liouville fractional integral \cite{goodrich0,lan2,lan3,podlubny1,webb0} of order $\alpha$.

The study of variable exponent functions has become an important area of research in recent decades, driven, in addition to their intrinsic mathematical interest, by their ability to model phenomena with non-standard growth conditions.  Applications of variable exponent functions have been explored in various fields. For example, Rajagopal and R\r{u}\v{z}i\v{c}ka \cite{rajagopal1} demonstrated how these functions can effectively model the behaviour of electrorheological fluids, where the viscosity changes in response to an electric field. The variable exponent allows for capturing the fluid’s non-Newtonian behaviour, which cannot be accurately represented by models with constant exponents.  Zhikov \cite{zhikov1,zhikov4,zhikov2,zhikov3} also studied such problems.  They have also been well studied within the context of regularity theory -- see, for example, Ragusa and Tachikawa \cite{ragusa4,ragusatachikawa1}.  And more recently Garc\'{i}a-Huidobro, et al. \cite{garcia1} have considered a class of \underline{local} one-dimensional boundary value problems with variable exponents.

At the same time, there is a vast literature on nonlocal ordinary and partial differential equations.  As with variable growth, part of the interest in nonlocal differential equations is their potential applications -- e.g., applications to beam deflection \cite{infantepietramala7}, chemical reactor theory \cite{infante11}, and thermodynamics \cite{cabada1}.  Within this context, two commonly studied problems are
\begin{equation}
-A\left(\Vert u\Vert_{L^p}^{p}\right)\Delta u(\bm{x})=\lambda f(\bm{x},u(\bm{x})\big)\text{, }\bm{x}\in\Omega\subset\mathbb{R}^{n}\notag
\end{equation}
and
\begin{equation}
-A\left(\Vert Du\Vert_{L^p}^{p}\right)\Delta u(\bm{x})=\lambda f(\bm{x},u(\bm{x})\big)\text{, }\bm{x}\in\Omega\subset\mathbb{R}^{n},\notag
\end{equation}
together with their one-dimensional equivalents.  Such problems are Kirchhoff-like equations, given their obvious connection to the classical Kirchhoff equation mentioned earlier.  A very common assumption \cite{afrouzi1,alves1,azzouz1,bellamouchi1,biagi1,boulaaras0,boulaaras1,chung1,correa1,correa2,do4,graef2,infante0,infante2,infante99,li1,stanczy1,wang1,yanma1,yan1} in such problems is to assume that the nonlocal coefficient, $A$, is strictly positive -- i.e., $A(t)>0$ for $t\ge0$; ostensibly this is a perfectly reasonable assumption so as to avoid degeneracy in the differential equation.  At the same time, rarely authors have deployed alternative assumptions.  For example, Ambrosetti and Arcoya \cite{ambrosetti1} allow $A$ to vanish at $0$, whereas Delgado, et al. \cite{delgado1} allow $A$ to vanish at one particular point, which need not be zero.  Finally, Santos J\'{u}nior and Siciliano \cite{santos1}, in the setting of an $L^2$ nonlocal element, require that $A$ be nonzero on a neighbourhood of zero.

On the other hand, over the past four years the first author, beginning with \cite{goodrich8} and further developed in \cite{goodrich10,goodrich12,goodrich15,goodrich16,goodrich20}, has developed a new methodology for the analysis of the existence of solution to one-dimensional nonlocal differential equations, with additional extensions having been provided by the first author and Lizama \cite{goodrichlizama3}, Hao and Wang \cite{hao1}, Shibata \cite{shibata3}, and Song and Hao \cite{song1}; some related nonexistence results may be found in \cite{goodrich19,goodrich24,shibata1,shibata2}.  Unlike the methodologies referenced in the previous paragraph, the new methodology requires only that $A$ be positive on a set of possibly very small measure with no \emph{a priori} location restriction -- e.g., the nonlocal element need not be nonzero on a neighbourhood of zero.  Moreover, the methodology easily accommodates a variety of nonlocal elements by way of finite convolution.

We note that all of the nonlocal differential equations papers referenced thus far treat \emph{constant} growth nonlocal equations.  By this we mean that the results concern problems similar to \eqref{eq1.1} but with $p(t)\equiv p_0$.  A natural question, then, given the interest in variable growth problems, is whether the methodology developed in \cite{goodrich8} can be extended to the variable growth setting.  This is not a trivial question, in fact, because the methods one usually uses in variable growth problems, say, for example, in regularity theory, do not naturally carry over to \eqref{eq1.1}.  In regularity theory, for instance, one typically works on sets sufficiently small in measure such that the exponent $p$ does not vary much.  Then one can then approximate $p(x)$ by a constant, via some regularity assumption on $p$ such as H\"{o}lder continuity, and thus analyse a constant exponent problem for which better estimates are available -- cf., \cite[(3.10)]{goodrichragusa1}.  In the setting of \eqref{eq1.1} one cannot do this because the estimates we require must be valid on the \emph{entire} interval $[0,1]$ not simply some very small subset of it.  So, there is no \emph{a priori} way to import the variable exponent techniques from regularity theory to the setting of problem \eqref{eq1.1}.

In \cite{goodrich23} the first author made an initial attempt to study \eqref{eq1.1}, though only in the setting where $1<p^-\le p(t)\le p^+<+\infty$.  There he imposed conditions which ensured that $\Vert u\Vert_{\infty}\ge1$.  In so doing, he could then pass from $\Vert u\Vert_{\infty}^{p(t)}$ to a constant exponent via the inequality $1\le\Vert u\Vert_{\infty}^{p^-}\le\Vert u\Vert_{\infty}^{p(t)}\le\Vert u\Vert_{\infty}^{p^+}$, and this was sufficient to obtain an existence result for the problem.  However, this methodology required making some undesirable assumptions (e.g., an extra growth assumption on the nonlinearity $f$) in order to ensure that $\Vert u\Vert_{\infty}\ge1$.

Very recently, the first author \cite{goodrich27} extended the methodology of \cite{goodrich23} to the Kirchhoff-type problem
\begin{equation}
-A\left(\left(b*|u'|^{p(\cdot)}\right)(1)\right)u''(t)=\lambda f\big(t,u(t)\big)\text{, }t\in(0,1),\notag
\end{equation}
which is \eqref{eq1.1} but with $|u'|$ replacing $u$; as in \cite{goodrich23}, the existence results of \cite{goodrich27} assumed that $p(x)>1$.  An important novelty introduced in \cite{goodrich27} was an inequality (see Lemma \ref{lemma2.1}) for passing from the variable exponent to a corresponding constant exponent.  This allowed for more refined estimates than were available in \cite{goodrich23}.

In this paper we aim to use the technique introduced in \cite{goodrich27} in order to refine and extend considerably the results of \cite{goodrich23}.  In particular, as mentioned earlier, we treat not only the convex-type setting in which $1<p(x)<+\infty$, but also the concave-type case in which $0<p(x)\le1$ as well as the mixed-type case in which $0<p(x)<+\infty$.  In addition, even in the case $1<p(x)<+\infty$, which explicitly overlaps with \cite{goodrich23}, we are able to improve the results of \cite{goodrich23} by removing some of the assumptions imposed there (e.g., an extra growth assumption imposed on $f$).

We conclude by mentioning the outline of the remainder of this paper.  We begin in Section 2 by first describing our notation and other necessary preliminaries.  Then we develop an existence theory for \eqref{eq1.1} in case $p(x)>1$, which is the convex-like case.  In Section 3 we develop a parallel theory for \eqref{eq1.1} in case $0<p(x)\le 1$, which is the concave-like case.  Finally, in Section 4 we demonstrate that by combining the preceding results we can actually allow for $0<p(x)<+\infty$, which is to say an exponent that can freely shift between the concave and convex regimes.  Consequently, the totality of our results is a unified theory for one-dimensional variable growth nonlocal equations having the form \eqref{eq1.1}.  We further emphasise that since our results subsume the constant exponent case, we also provide a unified theory for that setting as well.

Finally, let us give a brief explanation as to why we separate the results into the three regimes outlined in the previous paragraph.  The reason is twofold.  First, since the most general case (i.e., Section 4) builds off of the more restrictive cases (i.e., Sections 2 and 3), there is a certain logic to organising the manuscript in this way.  But, secondly, if one has $1<p(x)<+\infty$, then the more specialised results of Section 2 will be better since we are able to give more specific estimates with results that are tailored to that particular regime.  So, there are also good mathematical reasons to structure the results as we have, even though at first glance there might seem to be a certain inefficiency.

\section{Preliminaries and Existence Theory for \eqref{eq1.1} in Case $p(x)>1$}

We begin by mentioning the notation that we use throughout the remainder of this paper.  By $\mathscr{C}\big([0,1]\big)$ we denote the space of continuous functions on $[0,1]$.  We always assume that this space is equipped with the norm $\Vert\cdot\Vert_{\infty}$, which is the usual maximum norm on $\mathscr{C}\big([0,1]\big)$.  So normed, the space $\mathscr{C}\big([0,1]\big)$ is Banach.  In addition, by $\bm{1}$ we denote the constant function $\bm{1}\ : \ \mathbb{R}\rightarrow\{1\}$, and likewise $\bm{0}\ : \ \mathbb{R}\rightarrow\{0\}$.  By $*$ we denote the final convolution functional on $[0,1]$ so that
\begin{equation}
(u*v)(t):=\int_0^tu(t-s)v(s)\ ds\text{, }0\le t\le 1,\notag
\end{equation}
for $u$, $v\in L^1\big((0,1)\big)$.  Finally, for a given continuous function $h\ : \ [0,1]\times[0,+\infty)\rightarrow[0,+\infty)$ and real numbers $0\le a<b\le1$ and $0\le c<d<+\infty$ we denote by $h_{[a,b]\times[c,d]}^{m}$ and $h_{[a,b]\times[c,d]}^{M}$ the following quantities.
\begin{equation}
\begin{split}
h_{[a,b]\times[c,d]}^{m}&:=\min_{(t,u)\in[a,b]\times[c,d]}h(t,u)\\
h_{[a,b]\times[c,d]}^{M}&:=\max_{(t,u)\in[a,b]\times[c,d]}h(t,u)\notag
\end{split}
\end{equation}

We next describe the general assumptions we use when studying equation \eqref{eq1.1}.  Note that in condition (H3) the function $G$ is a Green's function, which will encode boundary data that a solution of \eqref{eq1.1} will be required to satisfy.  For example, if $G$ is defined by
\begin{equation}\label{eq2.1mmm}
G(t,s):=\begin{cases} t(1-s)\text{, }&0\le t\le s\le 1\\ s(1-t)\text{, }&0\le s\le t\le 1\end{cases},
\end{equation}
then $G$ encodes Dirichlet boundary data.  Similarly, if
\begin{equation}\label{eq2.2ppp}
G(t,s):=\begin{cases} t\text{, }&0\le t\le s\le 1\\ s\text{, }&0\le s\le t\le 1\end{cases},
\end{equation}
then $G$ encodes right-focal boundary data.  Please see Erbe and Wang \cite{erbe1} for additional details.

\begin{list}{}{\setlength{\leftmargin}{.5in}\setlength{\rightmargin}{0in}}
\item[\textbf{H1:}] The functions $A\ : \ [0,+\infty)\rightarrow\mathbb{R}$, $f\ : \ [0,1]\times[0,+\infty)\rightarrow[0,+\infty)$, and $b\ : \ (0,1]\rightarrow[0,+\infty)$ satisfy the following properties.
\begin{enumerate}
\item Each of $A$ and $f$ is continuous on their respective domain.

\item Both $b\in L^1\big((0,1];[0,+\infty)\big)$ and $(b*\bm{1})(1)\neq0$.

\item There exist numbers $0<\rho_1<\rho_2$ such that $A(t)>0$ for each $t\in\big[\rho_1,\rho_2\big]$.
\end{enumerate}

\item[\textbf{H2:}] The function $p\ : \ [0,1]\rightarrow(1,+\infty)$ is continuous, and there exist real numbers $p^-$ and $p^+$ such that, for each $t\in[0,1]$,
\begin{equation}
1<p^-\le p(t)\le p^+<+\infty.\notag
\end{equation}

\item[\textbf{H3:}] The continuous function $G\ : \ [0,1]\times[0,1]\rightarrow[0,+\infty)$ satisfies the following properties.
\begin{enumerate}
\item Denote by $\mathscr{G}\ : \ [0,1]\rightarrow[0,+\infty)$ the function defined by
\begin{equation}
\mathscr{G}(s):=\max_{t\in[0,1]}G(t,s).\notag
\end{equation}
There exist numbers $0\le\alpha<\beta\le1$ and a number $\eta_0:=\eta_0(\alpha,\beta)\in(0,1]$ such that
\begin{equation}
\min_{t\in[\alpha,\beta]}G(t,s)\ge\eta_0\mathscr{G}(s)\text{, for each }s\in[0,1].\notag
\end{equation}

\item The number $C_0$ defined by
\begin{equation}
C_0:=\sup_{s\in(0,1)}\frac{1}{\mathscr{G}(s)}\int_0^1G(t,s)\ dt\notag
\end{equation}
satisfies $C_0>0$.
\end{enumerate}
\end{list}
We note that both the Green's functions \eqref{eq2.1mmm}--\eqref{eq2.2ppp} mentioned above, together with many others, satisfy condition (H3) -- see, for example, \cite{erbe1,goodrich5}.

With our general assumptions (H1)--(H3) in hand, we conclude the preliminaries by describing the functional analytic framework in which we study \eqref{eq1.1}.  To this end, denote by $\mathscr{K}\subseteq\mathscr{C}\big([0,1]\big)$ the positive order cone
\begin{equation}
\mathscr{K}:=\Big\{u\in\mathscr{C}\big([0,1]\big)\ : \ u\ge0\text{, }\min_{t\in[\alpha,\beta]}u(t)\ge\eta_0\Vert u\Vert_{\infty}\text{, and }(\bm{1}*u)(1)\ge C_0\Vert u\Vert_{\infty}\Big\}.\notag
\end{equation}
Attendant to the cone $\mathscr{K}$, for $\rho>0$ we denote by $\widehat{V}_{\rho}\subseteq\mathscr{K}$ the (relatively) open set
\begin{equation}
\widehat{V}_{\rho}:=\left\{u\in\mathscr{K}\ : \ \left(b*u^{p(\cdot)}\right)(1)<\rho\right\}\subseteq\mathscr{K}.\notag
\end{equation}
As will be used repeatedly in the sequel, we observe that
\begin{equation}
\partial\widehat{V}_{\rho}=\left\{u\in\mathscr{K}\ : \ \left(b*u^{p(\cdot)}\right)(1)=\rho\right\}.\notag
\end{equation}
We note that $\widehat{V}_{\rho}$ was introduced in \cite{goodrich23}.  Finally, because we will study \eqref{eq1.1} via a topological fixed point approach, we define the operator $T$ by
\begin{equation}
(Tu)(t):=\lambda\int_0^1\left(A\left(\left(b*u^{p(\cdot)}\right)(1)\right)\right)^{-1}G(t,s)f\big(s,u(s)\big)\ ds,\notag
\end{equation}
and we note that a nontrivial fixed point of $T$
\begin{enumerate}
\item is a nontrivial solution of \eqref{eq1.1};

\item satisfies the boundary data encoded by the Green's function $G$; and so

\item is a positive solution of the problem defined by (1)--(2).
\end{enumerate}
As will be specified precisely in the existence theorems to follow, we will restrict $T$ to a closed annular subset of $\mathscr{K}$ on which $\displaystyle A\left(\left(b*u^{p(\cdot)}\right)(1)\right)>0$.  In this way, then, $T$ will be well defined.

With the preliminaries dispatched, our first lemma is a pointwise estimate that allows us to switch from variable exponent growth to constant exponent growth.  This lemma, which was proved by the first author in \cite{goodrich27}, states a fundamental pointwise relationship between a function taken to a variable exponent and the same function taken to a constant exponent.  The lemma may be thought of as a global version of the local variable-to-constant exponent estimates used in regularity theory -- cf., \cite{goodrichragusa1}.  The proof of this lemma may be found in \cite[Lemma 2.6]{goodrich27}.

\begin{lemma}\label{lemma2.1}
Let $f\ : \ [0,1]\rightarrow[0,+\infty)$ be given.  Suppose that $p\ : \ [0,1]\rightarrow(1,+\infty)$ satisfies condition (H2).  Then, given any constant $q$ satisfying $1\le q<p^-$, for each $t\in[0,1]$ it holds that
\begin{equation}
\big(f(t)\big)^{\frac{p(t)}{q}}\ge 2^{1-\frac{p^+}{q}}\big(f(t)\big)^{\frac{p^-}{q}}-1.\notag
\end{equation}
\end{lemma}

The proof of our next result may be found in \cite[Corollary 2.7]{goodrich27}.

\begin{corollary}\label{corollary2.2}
Suppose that the hypotheses of Lemma \ref{lemma2.1} are true.  Then under the additional assumption that $f\in\mathscr{C}\big([0,1];[0,+\infty)\big)$ it holds that
\begin{equation}
\begin{split}
\int_0^1\big(f(t)\big)^{\frac{p(t)}{q}}\ dt&\ge2^{1-\frac{p^+}{q}}\int_0^1\big(f(t)\big)^{\frac{p^-}{q}}\ dt-1\\
&=2^{1-\frac{p^+}{q}}\Vert f\Vert_{L^{\frac{p^-}{q}}}^{\frac{p^-}{q}}-1.\notag
\end{split}
\end{equation}
\end{corollary}

Our next result, which is our first original result of this paper, establishes a lower bound on $\Vert u\Vert_{\infty}$ provided that $u\in\partial\widehat{V}_{\rho}$.  Such a result in the variable growth problem was originally reported in \cite[Lemma 2.4]{goodrich23}.  Our result, Lemma \ref{lemma2.3aaa}, establishes the same lower bound as \cite[Lemma 2.4]{goodrich23}.  However, Lemma \ref{lemma2.3aaa} shows that the lower bound can be written in a different way, which does not require the use of the minimum of two quantities as in \cite{goodrich23}.  And this will be useful in certain of the estimates we wish to use later.

\begin{lemma}\label{lemma2.3aaa}
Suppose that $u\in\partial\widehat{V}_{\rho}$ for some $\rho>0$.  Then
\begin{equation}
\Vert u\Vert_{\infty}\ge\left(\frac{\rho}{(b*\bm{1})(1)}\right)^{\frac{1}{p^+}}+\varepsilon(\rho,b),\notag
\end{equation}
where
\begin{equation}
\varepsilon(\rho,b):=\begin{cases} \left(\frac{\rho}{(b*\bm{1})(1)}\right)^{\frac{1}{p^-}}-\left(\frac{\rho}{(b*\bm{1})(1)}\right)^{\frac{1}{p^+}}\text{, }&0<\left(\frac{\rho}{(b*\bm{1})(1)}\right)^{\frac{1}{p^+}}<1\\ 0\text{, }&\left(\frac{\rho}{(b*\bm{1})(1)}\right)^{\frac{1}{p^+}}\ge1\end{cases}.\notag
\end{equation}
\end{lemma}

\begin{proof}
Suppose first that
\begin{equation}\label{eq2.1}
0<\left(\frac{\rho}{(b*\bm{1})(1)}\right)^{\frac{1}{p^+}}<1.
\end{equation}
We consider cases depending upon the magnitude of $\Vert u\Vert_{\infty}$.  So, first suppose that $\Vert u\Vert_{\infty}\le1$.  Then
\begin{equation}\label{eq2.2}
\rho=\left(b*u^{p(\cdot)}\right)(1)\le\left(b*\Vert u\Vert_{\infty}^{p(\cdot)}\right)(1)\le\left(b*\Vert u\Vert_{\infty}^{p^-}\bm{1}\right)(1).
\end{equation}
But then \eqref{eq2.2}, combined with the definition of the function $\varepsilon$ in the statement of the lemma, implies that
\begin{equation}\label{eq2.3}
\Vert u\Vert_{\infty}\ge\left(\frac{\rho}{(b*\bm{1})(1)}\right)^{\frac{1}{p^-}}=\left(\frac{\rho}{(b*\bm{1})(1)}\right)^{\frac{1}{p^+}}+\varepsilon(\rho,b).
\end{equation}
On the other hand, in case $\Vert u\Vert_{\infty}>1$, recalling assumption \eqref{eq2.1}, we see that
\begin{equation}\label{eq2.5}
\Vert u\Vert_{\infty}>1>\underbrace{\left(\frac{\rho}{(b*\bm{1})(1)}\right)^{\frac{1}{p^+}}}_{\in(0,1)}\ge\left(\frac{\rho}{(b*\bm{1})(1)}\right)^{\frac{1}{p^-}}.
\end{equation}
Thus, inequality \eqref{eq2.5} implies that
\begin{equation}\label{eq2.7}
\Vert u\Vert_{\infty}\ge\left(\frac{\rho}{(b*\bm{1})(1)}\right)^{\frac{1}{p^+}}+\varepsilon(\rho,b).
\end{equation}
So, in light of inequalities \eqref{eq2.3} and \eqref{eq2.7}, we conclude that in case \eqref{eq2.1} holds, the claimed formula is true.

Next we consider the case
\begin{equation}\label{eq2.8}
\left(\frac{\rho}{(b*\bm{1})(1)}\right)^{\frac{1}{p^+}}\ge1.
\end{equation}
Note, again, that
\begin{equation}\label{eq2.9}
\rho=\left(b*u^{p(\cdot)}\right)(1)\le\left(b*\Vert u\Vert_{\infty}^{p(\cdot)}\right)(1).
\end{equation}
Now, if $\Vert u\Vert_{\infty}<1$, then \eqref{eq2.9} implies that
\begin{equation}
\rho=\left(b*u^{p(\cdot)}\right)(1)\le\left(b*\Vert u\Vert_{\infty}^{p(\cdot)}\right)(1)\le\left(b*\Vert u\Vert_{\infty}^{p^-}\bm{1}\right)(1),\notag
\end{equation}
from which it follows that
\begin{equation}\label{eq2.10}
\Vert u\Vert_{\infty}^{p^-}\ge\frac{\rho}{(b*\bm{1})(1)}.
\end{equation}
But then inequality \eqref{eq2.10}, recalling inequality \eqref{eq2.8}, implies that
\begin{equation}
\Vert u\Vert_{\infty}\ge\left(\frac{\rho}{(b*\bm{1})(1)}\right)^{\frac{1}{p^-}}\ge\left(\frac{\rho}{(b*\bm{1})(1)}\right)^{\frac{1}{p^+}}\ge1,\notag
\end{equation}
which is a contradiction.  Therefore, we conclude that if inequality \eqref{eq2.8} holds, then $\Vert u\Vert_{\infty}\ge1$.  But since $\Vert u\Vert_{\infty}\ge1$, it follows that
\begin{equation}
\begin{split}
\rho=\left(b*u^{p(\cdot)}\right)(1)&\le\left(b*\Vert u\Vert_{\infty}^{p(\cdot)}\bm{1}\right)(1)\\
&\le\left(b*\Vert u\Vert_{\infty}^{p^+}\bm{1}\right)(1)\\
&=\Vert u\Vert_{\infty}^{p^+}(b*\bm{1})(1),\notag
\end{split}
\end{equation}
and so,
\begin{equation}
\Vert u\Vert_{\infty}\ge\left(\frac{\rho}{(b*\bm{1})(1)}\right)^{\frac{1}{p^+}},\notag
\end{equation}
which verifies the claimed formula in case \eqref{eq2.8} holds.

Since the preceding cases are exhaustive, we conclude that whenever $u\in\partial\widehat{V}_{\rho}$, it follows that
\begin{equation}
\Vert u\Vert_{\infty}\ge\left(\frac{\rho}{(b*\bm{1})(1)}\right)^{\frac{1}{p^+}}+\varepsilon(\rho,b),\notag
\end{equation}
with $\varepsilon(\rho,b)$ defined as in the statement of the lemma.  And this completes the proof of the lemma.
\end{proof}

Our next lemma establishes an upper bound on $\Vert u\Vert_{\infty}$ whenever $u$ is in either $\widehat{V}_{\rho}$ or its boundary.

\begin{lemma}\label{lemma2.3}
Assume that, for some $q\in\big(1,p^-\big)$,
\begin{equation}
b^{\frac{1}{1-q}}\in L^1\big((0,1]\big).\notag
\end{equation}
For each $\rho>0$, if $u\in\widehat{V}_{\rho}$, then
\begin{equation}
\Vert u\Vert_{\infty}<C_0^{-1}2^{\frac{p^{+}-q}{p^-}}\left[\rho^{\frac{1}{q}}\left(\left(b^{\frac{1}{1-q}}*\bm{1}\right)(1)\right)^{\frac{q-1}{q}}+1\right]^{\frac{q}{p^-}}.\notag
\end{equation}
Furthermore, if $u\in\partial\widehat{V}_{\rho}$, then
\begin{equation}
\Vert u\Vert_{\infty}\le C_0^{-1}2^{\frac{p^{+}-q}{p^-}}\left[\rho^{\frac{1}{q}}\left(\left(b^{\frac{1}{1-q}}*\bm{1}\right)(1)\right)^{\frac{q-1}{q}}+1\right]^{\frac{q}{p^-}}.\notag
\end{equation}
\end{lemma}

\begin{proof}
Let $\rho>0$ be fixed but arbitrary and suppose that $u\in\widehat{V}_{\rho}$.  Then
\begin{equation}\label{eq2.11}
\left(b*u^{p(\cdot)}\right)(1)<\rho.
\end{equation}
Now, both by Lemma \ref{lemma2.1} and the reverse H\"{o}lder inequality we see that
\begin{equation}\label{eq2.12}
\begin{split}
\left(b*u^{p(\cdot)}\right)(1)&=\int_0^1b(1-s)\big(u(s)\big)^{p(s)}\ ds\\
&\ge\left(\int_0^1\big(b(1-s)\big)^{\frac{1}{1-q}}\ ds\right)^{1-q}\left(\int_0^1\big(u(s)\big)^{\frac{p(s)}{q}}\ ds\right)^{q}\\
&\ge\left(\int_0^1\big(b(1-s)\big)^{\frac{1}{1-q}}\ ds\right)^{1-q}\left[\int_0^1\left(2^{1-\frac{p^+}{q}}\left(u(s)\right)^{\frac{p^-}{q}}-1\right)\ ds\right]^q\\
&=\left(\left(b^{\frac{1}{1-q}}*\bm{1}\right)(1)\right)^{1-q}\left[-1+2^{1-\frac{p^+}{q}}\int_0^1\big(u(s)\big)^{\frac{p^-}{q}}\ ds\right]^q.
\end{split}
\end{equation}
Since
\begin{equation}
\int_0^1\big(u(s)\big)^{\frac{p^-}{q}}\ ds\ge\left(\int_0^1u(s)\ ds\right)^{\frac{p^-}{q}}\notag
\end{equation}
by means of Jensen's inequality, it follows from \eqref{eq2.12} that
\begin{equation}\label{eq2.13}
\begin{split}
\left(b*u^{p(\cdot)}\right)(1)&\ge\left(\left(b^{\frac{1}{1-q}}*\bm{1}\right)(1)\right)^{1-q}\left[-1+2^{1-\frac{p^+}{q}}\int_0^1\big(u(s)\big)^{\frac{p^-}{q}}\ ds\right]^q\\
&\ge\left(\left(b^{\frac{1}{1-q}}*\bm{1}\right)(1)\right)^{1-q}\left[-1+2^{1-\frac{p^+}{q}}\left(\int_0^1u(s)\ ds\right)^{\frac{p^-}{q}}\right]^q\\
&\ge\left(\left(b^{\frac{1}{1-q}}*\bm{1}\right)(1)\right)^{1-q}\left[-1+2^{1-\frac{p^+}{q}}C_0^{\frac{p^-}{q}}\Vert u\Vert_{\infty}^{\frac{p^-}{q}}\right]^q,
\end{split}
\end{equation}
where we have used the coercivity relation
\begin{equation}
\int_0^1u(s)\ ds=(u*\bm{1})(1)\ge C_0\Vert u\Vert_{\infty},\notag
\end{equation}
seeing as $u\in\mathscr{K}$.  Therefore, combining both \eqref{eq2.11} and \eqref{eq2.13} we arrive at
\begin{equation}
\rho>\left(\left(b^{\frac{1}{1-q}}*\bm{1}\right)(1)\right)^{1-q}\left[-1+2^{1-\frac{p^+}{q}}C_0^{\frac{p^-}{q}}\Vert u\Vert_{\infty}^{\frac{p^-}{q}}\right]^q,\notag
\end{equation}
from which it follows that
\begin{equation}
\Vert u\Vert_{\infty}^{\frac{p^-}{q}}<2^{\frac{p^+}{q}-1}C_0^{-\frac{p^-}{q}}\left[\rho^{\frac{1}{q}}\left(\left(b^{\frac{1}{1-q}}*\bm{1}\right)(1)\right)^{\frac{q-1}{q}}+1\right]\notag
\end{equation}
so that
\begin{equation}\label{eq2.14}
\Vert u\Vert_{\infty}<2^{\frac{p^+-q}{p^-}}C_0^{-1}\left[\rho^{\frac{1}{q}}\left(\left(b^{\frac{1}{1-q}}*\bm{1}\right)(1)\right)^{\frac{q-1}{q}}+1\right]^{\frac{q}{p^-}},
\end{equation}
as claimed.

On the other hand, if $u\in\partial\widehat{V}_{\rho}$, then an inspection of the preceding argument demonstrates that the only change under the new assumption is that the strict inequality deduced above in \eqref{eq2.14} changes to a non-strict inequality.  And this completes the proof.
\end{proof}

As a corollary to Lemma \ref{lemma2.3}, we obtain the following result.

\begin{corollary}\label{corollary2.4}
For each $\rho>0$ the set $\widehat{V}_{\rho}$ is bounded.
\end{corollary}

Our next lemma asserts that $T$ maps a solid annular subset of $\mathscr{K}$ back into $\mathscr{K}$, which is important for the topological fixed point result contained in Lemma \ref{lemma2.4aaa}.

\begin{lemma}\label{lemma2.6}
Assume that conditions (H1)--(H3) are true.  Then the operator $T$ is completely continuous on $\overline{\widehat{V}_{\rho_2}}\setminus\widehat{V}_{\rho_1}$.  Moreover,
\begin{equation}
T\left(\overline{\widehat{V}_{\rho_2}}\setminus\widehat{V}_{\rho_1}\right)\subseteq\mathscr{K}.\notag
\end{equation}
\end{lemma}

\begin{proof}
Omitted -- see, for example, \cite[Lemma 2.11]{goodrich22}.
\end{proof}

Finally, we state the topological fixed point theorem, which we use in the proof of our existence result -- see Cianciaruso, Infante, and Pietramala \cite[Lemma 2.1]{cianciaruso1}.

\begin{lemma}\label{lemma2.4aaa}
Let $U$ be a bounded open set and, with $\mathscr{K}$ a cone in a real Banach space $\mathscr{X}$, suppose both that $U_{\mathscr{K}}:=U\cap\mathscr{K}\supseteq\{\bm{0}\}$ and that $\overline{U_{\mathscr{K}}}\neq\mathscr{K}$.  Assume that $T\ : \ \overline{U_{\mathscr{K}}}\rightarrow\mathscr{K}$ is a compact map such that $x\neq Tx$ for each $x\in\partial U_{\mathscr{K}}$.  Then the fixed point index $i_{\mathscr{K}}\left(T,U_{\mathscr{K}}\right)$ has the following properties.
\begin{enumerate}
\item If there exists $v\in\mathscr{K}\setminus\{\bm{0}\}$ such that $x\neq Tx+\lambda v$ for each $x\in\partial U_{\mathscr{K}}$ and each $\lambda>0$, then $i_{\mathscr{K}}\left(T,U_{\mathscr{K}}\right)=0$.

\item If $\mu x\neq Tx$ for each $x\in\partial U_{\mathscr{K}}$ and for each $\mu\ge1$, then $i_{\mathscr{K}}\left(T,U_{\mathscr{K}}\right)=1$.

\item If $i_{\mathscr{K}}\left(T,U_{\mathscr{K}}\right)\neq0$, then $T$ has a fixed point in $U_{\mathscr{K}}$.

\item Let $U^1$ be open in $X$ with $\overline{U_{\mathscr{K}}^1}\subseteq U_{\mathscr{K}}$.  If $i_{\mathscr{K}}\left(T,U_{\mathscr{K}}\right)=1$ and $i_{\mathscr{K}}\left(T,U_{\mathscr{K}}^{1}\right)=0$, then $T$ has a fixed point in $U_{\mathscr{K}}\setminus\overline{U_{\mathscr{K}}^{1}}$.  The same result holds if $i_{\mathscr{K}}\left(T,U_{\mathscr{K}}\right)=0$ and $i_{\mathscr{K}}\left(T,U_{\mathscr{K}}^{1}\right)=1$.
\end{enumerate}
\end{lemma}

With the necessary preliminary results stated, we now prove an existence theorem for \eqref{eq1.1} subject to the boundary data implied by the Green's function $G$.  In both the statement and proof of Theorem \ref{theorem2.8} it will useful to use the following notation, where $\rho>0$ is a real number.
\begin{equation}
\begin{split}
G^M&:=\max_{\tau\in[0,1]}\int_0^1G(\tau,s)\ ds\\
m_{\rho}&:=\left(\frac{\rho}{(b*\bm{1})(1)}\right)^{\frac{1}{p^+}}+\varepsilon(\rho,b)\\
M_{\rho}&:=C_0^{-1}2^{\frac{p^{+}-q}{p^-}}\left[\rho^{\frac{1}{q}}\left(\left(b^{\frac{1}{1-q}}*\bm{1}\right)(1)\right)^{\frac{q-1}{q}}+1\right]^{\frac{q}{p^-}}\notag
\end{split}
\end{equation}
Note that the function $(\rho,b)\mapsto\varepsilon(\rho,b)$ appearing in $m_{\rho}$ is the same function as was defined in Lemma \ref{lemma2.3aaa} earlier.

\begin{theorem}\label{theorem2.8}
Suppose for some $q\in\big(1,p^-\big)$ that $b^{\frac{1}{1-q}}\in L^1\big((0,1]\big)$.  In addition, suppose that the numbers $\rho_1$ and $\rho_2$ from condition (H1.3) are such that each of the following inequalities holds.
\vskip0.2cm
\begin{enumerate}
\item $\displaystyle\left[2^{1-p^+}\left(\frac{\lambda f_{[\alpha,\beta]\times\left[\eta_0m_{\rho_1},M_{\rho_1}\right]}^{m}}{A\left(\rho_1\right)}\max_{\tau\in[0,1]}\int_{\alpha}^{\beta}G(\tau,s)\ ds\right)^{p^-}-1\right]\int_{\alpha}^{\beta}b(1-t)\ dt>\frac{\rho_1}{\eta_0^{p^+}}$
\vskip0.2cm
\item $\displaystyle\max\left\{\left(\frac{\lambda f_{[0,1]\times\left[0,M_{\rho_2}\right]}^{M}G^M}{A\left(\rho_2\right)}\right)^{p^-},\left(\frac{\lambda f_{[0,1]\times\left[0,M_{\rho_2}\right]}^{M}G^M}{A\left(\rho_2\right)}\right)^{p^+}\right\}\int_0^1b(1-t)\ dt<\rho_2$
\end{enumerate}
\vskip0.2cm
If each of conditions (H1)--(H3) holds, then problem \eqref{eq1.1}, subject to the boundary data encoded by the Green's function $G$, has at least one positive solution $u_0\in\widehat{V}_{\rho_2}\setminus\overline{\widehat{V}_{\rho_1}}$ satisfying the localisation
\begin{equation}
\left(\frac{\rho_1}{(b*\bm{1})(1)}\right)^{\frac{1}{p^+}}+\varepsilon\left(\rho_1,b\right)\le\Vert u\Vert_{\infty}\le C_0^{-1}2^{\frac{p^+-q}{p^-}}\left[\rho_2^{\frac{1}{q}}\left(\left(b^{\frac{1}{1-q}}*\bm{1}\right)(1)\right)^{\frac{q-1}{q}}+1\right]^{\frac{q}{p^-}}.\notag
\end{equation}
\end{theorem}

\begin{proof}
Let us begin by noting that $\bm{0}\in\widehat{V}_{\rho_2}$ since $\displaystyle\left(a*\bm{0}^{p(\cdot)}\right)(1)=0$; this is necessary for the proper application of Lemma \ref{lemma2.4aaa}.  We will first appeal to part (1) of Lemma \ref{lemma2.4aaa}, and so, we aim to demonstrate that whenever $u\in\partial\widehat{V}_{\rho_1}$ it follows that $u\not\equiv Tu+\mu\bm{1}$ for all $\mu>0$; note that $\bm{1}\in\mathscr{K}$, owing to the fact that $\eta_0$, $C_0\in(0,1]$.  So, for contradiction suppose that there exists $u\in\partial\widehat{V}_{\rho_1}$ and $\mu>0$ such that $u(t)=(Tu)(t)+\mu$ for each $t\in[0,1]$.  Then it follows that
\begin{equation}
\big(u(t)\big)^{p(t)}=\big((Tu)(t)+\mu\big)^{p(t)}\notag
\end{equation}
so that
\begin{equation}\label{eq2.15}
\begin{split}
\rho_1=\left(b*u^{p(\cdot)}\right)(1)&\ge\left(b*(Tu)^{p(\cdot)}\right)(1)\\
&\ge\int_{\alpha}^{\beta}b(1-t)\big(\eta_0\Vert Tu\Vert_{\infty}\big)^{p(t)}\ dt\\
&\ge\eta_0^{p^+}\int_{\alpha}^{\beta}b(1-t)\Vert Tu\Vert_{\infty}^{p(t)}\ dt\\
&\ge\eta_0^{p^+}\int_{\alpha}^{\beta}b(1-t)\left(2^{1-p^+}\Vert Tu\Vert_{\infty}^{p^-}-1\right)\ dt\\
&=\eta_0^{p^+}\left[2^{1-p^+}\Vert Tu\Vert_{\infty}^{p^-}-1\right]\int_{\alpha}^{\beta}b(1-t)\ dt,
\end{split}
\end{equation}
where we have used Lemma \ref{lemma2.1} (with $q=1$ in the statement of the lemma) to obtain the final inequality.  Observe, furthermore, that
\begin{equation}\label{eq2.16}
\begin{split}
\Vert Tu\Vert_{\infty}^{p^-}&=\left(\max_{\tau\in[0,1]}\lambda\int_0^1\big(A\left(\rho_1\right)\big)^{-1}G(\tau,s)f\big(s,u(s)\big)\ ds\right)^{p^-}\\
&=\left(\frac{\lambda}{A\left(\rho_1\right)}\max_{\tau\in[0,1]}\int_0^1G(\tau,s)f\big(s,u(s)\big)\ ds\right)^{p^-}\\
&\ge\left(\frac{\lambda}{A\left(\rho_1\right)}\max_{\tau\in[0,1]}\int_{\alpha}^{\beta}G(\tau,s)f\big(s,u(s)\big)\ ds\right)^{p^-}\\
&\ge\left(\frac{\lambda f_{[\alpha,\beta]\times\left[\eta_0m_{\rho_1},M_{\rho_1}\right]}^{m}}{A\left(\rho_1\right)}\max_{\tau\in[0,1]}\int_{\alpha}^{\beta}G(\tau,s)\ ds\right)^{p^-},
\end{split}
\end{equation}
whereupon using inequality \eqref{eq2.16} in inequality \eqref{eq2.15} we see that
\begin{equation}\label{eq2.17}
\begin{split}
\rho_1&\ge\eta_0^{p^+}\left[2^{1-p^+}\left(\frac{\lambda f_{[\alpha,\beta]\times\left[\eta_0m_{\rho_1},M_{\rho_1}\right]}^{m}}{A\left(\rho_1\right)}\max_{\tau\in[0,1]}\int_{\alpha}^{\beta}G(\tau,s)\ ds\right)^{p^-}-1\right]\int_{\alpha}^{\beta}b(1-t)\ dt.
\end{split}
\end{equation}
Note that in \eqref{eq2.16} we have used, by means of Lemmata \ref{lemma2.3aaa}--\ref{lemma2.3}, that
\begin{equation}
M_{\rho_1}\ge\Vert u\Vert_{\infty}\ge u(t)\ge\eta_0\Vert u\Vert_{\infty}\ge\eta_0m_{\rho_1},\notag
\end{equation}
for each $t\in[\alpha,\beta]$.  But then, in light of condition (1) in the statement of the theorem, inequality \eqref{eq2.17} implies a contradiction.  Therefore, from part (1) of Lemma \ref{lemma2.4aaa} we conclude that
\begin{equation}\label{eq2.18}
i_{\mathscr{K}}\left(T,\widehat{V}_{\rho_1}\right)=0.
\end{equation}

We next demonstrate that whenever $u\in\partial\widehat{V}_{\rho_2}$ and $\mu\ge1$ it follows that $Tu\not\equiv\mu u$, being as we aim to use part (2) of Lemma \ref{lemma2.4aaa}.  So, for contradiction suppose that there exists $u\in\partial\widehat{V}_{\rho_2}$ and $\mu\ge1$ such that $(Tu)(t)=\mu u(t)$ for each $t\in[0,1]$.  Then for each $0\le t\le 1$ it follows that
\begin{equation}\label{eq2.19bbb}
\big(\mu u(t)\big)^{p(t)}=\big((Tu)(t)\big)^{p(t)}.
\end{equation}
Now we consider cases.  If $\Vert Tu\Vert_{\infty}\ge1$, then \eqref{eq2.19bbb} leads to
\begin{equation}\label{eq2.19}
\begin{split}
\rho_2&=\int_0^1b(1-t)\big(u(t)\big)^{p(t)}\ dt\\
&\le\int_0^1b(1-t)\big(\mu u(t)\big)^{p(t)}\ dt\\
&=\int_0^1b(1-t)\big(Tu(t)\big)^{p(t)}\ dt\\
&\le\int_0^1b(1-t)\Vert Tu\Vert_{\infty}^{p^+}\ dt\\
&=\int_0^1b(1-t) \left[\max_{\tau\in[0,1]}\lambda\int_0^1\left( A \left( \rho_2 \right)\right)^{-1}G(\tau, s)f(s,u(s))\ ds\right]^{p^+}\ dt\\
&\le\left(\frac{\lambda f_{[0,1]\times\left[0,M_{\rho_2}\right]}^{M}}{A\left(\rho_2\right)}\max_{\tau\in[0,1]}\int_0^1G(\tau,s)\ ds\right)^{p^+}\int_0^1b(1-t)\ dt\\
&=\left(\frac{\lambda f_{[0,1]\times\left[0,M_{\rho_2}\right]}^{M}G^M}{A\left(\rho_2\right)}\right)^{p^+}\int_0^1b(1-t)\ dt,
`\end{split}
\end{equation}
which, in light of condition (2) in the statement of the theorem, is a contradiction.  On the other hand, if $\Vert Tu\Vert_{\infty}<1$, then \eqref{eq2.19bbb} leads, in a similar manner as \eqref{eq2.19}, to
\begin{equation}\label{eq2.21bbb}
\begin{split}
\rho_2&\le\int_0^1b(1-t)\big(Tu(t)\big)^{p(t)}\ dt\\
&\le\int_0^1b(1-t)\Vert Tu\Vert_{\infty}^{p^-}\ dt\\
&\le\left(\frac{\lambda f_{[0,1]\times\left[0,M_{\rho_2}\right]}^{M}G^M}{A\left(\rho_2\right)}\right)^{p^-}\int_0^1b(1-t)\ dt,
\end{split}
\end{equation}
which, in light of condition (2) in the statement of the theorem, is also a contradiction.  Therefore, from part (ii) of Lemma \ref{lemma2.4aaa} we conclude jointly from \eqref{eq2.19} and \eqref{eq2.21bbb} that
\begin{equation}\label{eq2.20}
i_{\mathscr{K}}\left(T,\widehat{V}_{\rho_2}\right)=1.
\end{equation}

All in all, combining both \eqref{eq2.18} and \eqref{eq2.20} we conclude from part (4) of Lemma \ref{lemma2.4aaa} that there exists
\begin{equation}
u_0\in\widehat{V}_{\rho_2}\setminus\overline{\widehat{V}_{\rho_1}}\notag
\end{equation}
such that $Tu_0\equiv u_0$, and so, $u_0$ is a positive solution of \eqref{eq1.1} subject to the boundary data implied by the Green's function $G$.  Moreover, an application of Lemmata \ref{lemma2.3aaa}--\ref{lemma2.3} implies that
\begin{equation}
\left(\frac{\rho_1}{(b*\bm{1})(1)}\right)^{\frac{1}{p^+}}+\varepsilon\left(\rho_1,b\right)\le\Vert u\Vert_{\infty}\le C_0^{-1}2^{\frac{p^+-q}{p^-}}\left[\rho_2^{\frac{1}{q}}\left(\left(b^{\frac{1}{1-q}}*\bm{1}\right)(1)\right)^{\frac{q-1}{q}}+1\right]^{\frac{q}{p^-}},\notag
\end{equation}
as claimed.  And this completes the proof of the theorem.
\end{proof}

Recall from Section 1 that in case $b\equiv\bm{1}$ we arrive at an important model case.  Indeed, in this case the nonlocal data assumes the form
\begin{equation}
A\left(\left(b*u^{p(\cdot)}\right)(1)\right)\equiv A\left(\left(\bm{1}*u^{p(\cdot)}\right)(1)\right)=A\left(\int_0^1\big(u(s)\big)^{p(s)}\ ds\right).\notag
\end{equation}
In addition, if $p\equiv p_0$, then a further important model case is obtained -- namely,
\begin{equation}
A\left(\left(b*u^{p(\cdot)}\right)(1)\right)\equiv A\left(\int_0^1\big(u(s)\big)^{p_0}\ ds\right)=A\left(\Vert u\Vert_{L^{p_0}}^{p_0}\right).\notag
\end{equation}
Corollaries \ref{corollary2.10}--\ref{corollary2.11} address these special cases by reformulating Theorem \ref{theorem2.8} in each model case.

Prior to stating these corollaries, however, we state and prove a technical lemma.  The upshot of this lemma is that it will help us provide a more refined localisation for the upper bound of $\Vert u\Vert_{\infty}$ in the statement of the corollaries.

\begin{lemma}\label{lemma2.9}
Define the function $\varphi\ : \ (1,p^-)\rightarrow(0,+\infty)$ by
\begin{equation}
\varphi(q):=C_0^{-1}2^{\frac{p^+-q}{p^-}}\left(\rho^{\frac{1}{q}}+1\right)^{\frac{q}{p^-}},\notag
\end{equation}
where $\rho>0$, $C_0$ is as defined earlier in this section, and $1<p^-<p^+<+\infty$.  Then $\varphi$ is a nonincreasing function on its domain.
\end{lemma}

\begin{proof}
First note that
\begin{equation}\label{eq2.23mmm}
\varphi'(q)=-\frac{1}{qp^-}C_0^{-1}\left(\rho^{\frac{1}{q}}+1\right)^{\frac{q}{p^-}-1}\left(\rho^{\frac{1}{q}}\ln{(\rho)}+q\left(\rho^{\frac{1}{q}}+1\right)\left(\ln{(2)}-\ln{\left(\rho^{\frac{1}{q}}+1\right)}\right)\right)2^{\frac{p^+-q}{p^-}}.
\end{equation}
Now, since
\begin{equation}
\min\left\{C_0^{-1},2^{\frac{p^+-q}{p^-}},\rho^{\frac{1}{q}},q,p^-\right\}>0,\notag
\end{equation}
it follows from \eqref{eq2.23mmm} that the sign of $\varphi'(q)$ is determined by the sign of
\begin{equation}
-\left(\rho^{\frac{1}{q}}\ln{(\rho)}+q\left(\rho^{\frac{1}{q}}+1\right)\left(\ln{(2)}-\ln{\left(\rho^{\frac{1}{q}}+1\right)}\right)\right).\notag
\end{equation}
Therefore, to prove that $\varphi$ is nonincreasing it suffices to prove that
\begin{equation}\label{eq2.24mmm}
\left(\rho^{\frac{1}{q}}\ln{(\rho)}+q\left(\rho^{\frac{1}{q}}+1\right)\left(\ln{(2)}-\ln{\left(\rho^{\frac{1}{q}}+1\right)}\right)\right)\ge0
\end{equation}
whenever $1<q<p^-$.

In order to establish inequality \eqref{eq2.24mmm} we consider cases based on the magnitude of $\rho$.  So, in case $\rho=1$, we note that
\begin{equation}
\left(1^{\frac{1}{q}}\ln{(1)}+q\left(1^{\frac{1}{q}}+1\right)\left(\ln{(2)}-\ln{\left(1^{\frac{1}{q}}+1\right)}\right)\right)=0.\notag
\end{equation}
Thus, inequality \eqref{eq2.24mmm} is trivially satisfied in this case.

Next consider the case in which $\rho>1$.  Then $\ln{\rho}>0$ and $\rho^{\frac{1}{q}}>1$.  Put
\begin{equation}
a:=\rho^{\frac{1}{q}}\notag
\end{equation}
so that $a>1$.  Then inequality \eqref{eq2.24mmm} is equivalent to
\begin{equation}\label{eq2.25mmm}
a\ln{(\rho)}+q(a+1)\big(\ln{2}-\ln{(a+1)}\big)\ge0.
\end{equation}
Note that inequality \eqref{eq2.25mmm} is true if and only if
\begin{equation}\label{eq2.26mmm}
a\ln{(a)}+(a+1)\ln{\frac{2}{a+1}}\ge0,
\end{equation}
using that
\begin{equation}
\ln{\rho}=q\ln{a}.\notag
\end{equation}
To prove that \eqref{eq2.26mmm} is true, define the function $h\ : \ (1,+\infty)\rightarrow\mathbb{R}$ by
\begin{equation}
h(a):=a\ln{(a)}+(a+1)\ln{\frac{2}{a+1}}.\notag
\end{equation}
Note that
\begin{equation}
h'(a)=\ln{\frac{2a}{a+1}}.\notag
\end{equation}
Now, since $a>1$, it follows that $2a>a+1$, and so, $\displaystyle\frac{2a}{a+1}>1$.  Therefore,
\begin{equation}\label{eq2.27mmm}
\ln{\frac{2a}{a+1}}>\ln{1}=0,
\end{equation}
whenever $a>1$.  But then inequality \eqref{eq2.27mmm} implies that $h'(a)>0$ for $a\in(1,+\infty)$.  Therefore, $h$ is strictly increasing on $(1,+\infty)$.  Moreover, observe that
\begin{equation}
\lim_{a\to1^+}h(a)=1\cdot\ln{(1)}+(1+1)\ln{\frac{2}{1+1}}=0.\notag
\end{equation}
Consequently, since $h$ is strictly increasing on $(0,+\infty)$ and $\displaystyle\lim_{a\to1^+}h(a)=0$, it follows that $h(a)>0$ for all $a>1$.  Thus, inequality \eqref{eq2.26mmm} is true, and so, we conclude that inequality \eqref{eq2.24mmm} is likewise true whenever $1<q<p^-$ and $\rho>1$.  And this completes the second case.

Finally, the third case occurs when $0<\rho<1$.  In this case note both that $\ln{\rho}<0$ and that $0<\rho^{\frac{1}{q}}<1$.  Put
\begin{equation}
b:=\rho^{\frac{1}{q}},\notag
\end{equation}
where $0<b<1$.  Then inequality \eqref{eq2.24mmm} becomes
\begin{equation}\label{eq2.28mmm}
b\ln{(\rho)}+q(b+1)\big(\ln{2}-\ln{(b+1)}\big)\ge0.
\end{equation}
Similar to the previous case, the fact that
\begin{equation}
\ln{\rho}=q\ln{b}\notag
\end{equation}
allows us to recast inequality \eqref{eq2.28mmm} in the equivalent form
\begin{equation}\label{eq2.29mmm}
b\ln{(b)}+(b+1)\ln{\frac{2}{b+1}}\ge0.
\end{equation}
Now, define the function $m\ : \ (0,1)\rightarrow\mathbb{R}$ by
\begin{equation}
m(b):=b\ln{(b)}+(b+1)\ln{\frac{2}{b+1}}.\notag
\end{equation}
Note that
\begin{equation}
m'(b)=\ln{\frac{2b}{b+1}}.\notag
\end{equation}
Since $0<b<1$, we notice that
\begin{equation}
0<\frac{2b}{b+1}<1,\notag
\end{equation}
from which it follows that
\begin{equation}\label{eq2.30mmm}
\ln{\frac{2b}{b+1}}<\ln{1}=0.
\end{equation}
But then inequality \eqref{eq2.30mmm} implies that $m'(b)<0$ for $b\in(0,1)$, and so, $m$ is strictly decreasing on the interval $(0,1)$.  Since, in addition, we see that
\begin{equation}
\lim_{b\to1^-}m(b)=0,\notag
\end{equation}
it follows that $m(b)>0$ for all $0<b<1$.  Consequently, inequality \eqref{eq2.29mmm} is true for all $0<b<1$, and so, inequality \eqref{eq2.28mmm} is likewise true.  And this means that inequality \eqref{eq2.24mmm} holds whenever $0<\rho<1$.  So, the third case is dispatched.

Since the preceding three cases are exhaustive, we conclude that inequality \eqref{eq2.24mmm} holds for all $1<q<p^-$ and all $\rho>0$.  And this means that
\begin{equation}
\varphi'(q)\le0\notag
\end{equation}
for all $1<q<p^-$.  Therefore, $\varphi$ is a nonincreasing function on its domain, as claimed.
\end{proof}

With Lemma \ref{lemma2.9} in hand, we now state the promised corollaries to Theorem \ref{theorem2.8}.  In the statement of each of the corollaries, we use the notation
\begin{equation}
M_{\rho}^{*}:=C_0^{-1}2^{\frac{p^+-p^-}{p^-}}\left(\rho^{\frac{1}{p^-}}+1\right)\notag
\end{equation}
for $\rho>0$.

\begin{corollary}\label{corollary2.10}
Suppose that $b\equiv\bm{1}$.  In addition, suppose that the numbers $\rho_1$ and $\rho_2$ from condition (H1.3) satisfy the following inequalities.
\vskip0.2cm
\begin{enumerate}
\item $\displaystyle\eta_0^{p^+}\left[2^{1-p^+}\left(\frac{\lambda f_{[\alpha,\beta]\times\left[\eta_0m_{\rho_1},M_{\rho_1}^*\right]}^{m}}{A\left(\rho_1\right)}\max_{\tau\in[0,1]}\int_{\alpha}^{\beta}G(\tau,s)\ ds\right)^{p^-}-1\right]>\frac{\rho_1}{\beta-\alpha}$
\vskip0.2cm
\item $\displaystyle\max\left\{\left(\frac{\lambda f_{[0,1]\times\left[0,M_{\rho_2}^*\right]}^{M}G^M}{A\left(\rho_2\right)}\right)^{p^-},\left(\frac{\lambda f_{[0,1]\times\left[0,M_{\rho_2}^*\right]}^{M}G^M}{A\left(\rho_2\right)}\right)^{p^+}\right\}<\rho_2$
\end{enumerate}
\vskip0.2cm
If each of conditions (H1)--(H3) holds, then problem \eqref{eq1.1}, subject to the boundary data encoded by the Green's function $G$, has at least one positive solution $u_0\in\widehat{V}_{\rho_2}\setminus\overline{\widehat{V}_{\rho_1}}$ satisfying the localisation
\begin{equation}
\rho_1^{\frac{1}{p^+}}+\varepsilon\left(\rho_1,\bm{1}\right)\le\Vert u_0\Vert_{\infty}\le C_0^{-1}2^{\frac{p^+-p^-}{p^-}}\left(\rho_2^{\frac{1}{p^-}}+1\right).\notag
\end{equation}
\end{corollary}

\begin{proof}
We only mention how one obtains, on the one hand, the upper bound in the localisation of $\Vert u_0\Vert_{\infty}$ and, on the other hand, the new right-endpoint on the second set in both $f^m$ and $f^M$.  So, note that because $b\equiv\bm{1}$, it follows that $b^{\frac{1}{1-q}}\in L^{1}\big((0,1)\big)$ for any $q\neq1$.  In particular, $b^{\frac{1}{1-q}}\in L^{1}\big((0,1)\big)$ for every $1<q<p^-$.  Now, note that
\begin{equation}\label{eq2.21}
q\mapsto C_0^{-1}2^{\frac{p^+-q}{p^-}}\left(\rho_2^{\frac{1}{q}}+1\right)^{\frac{q}{p^-}}
\end{equation}
is a decreasing function of $q$, owing to the conclusion of Lemma \ref{lemma2.9}.  But this means that
\begin{equation}\label{eq2.22}
\inf_{q\in\left(1,p^-\right)}C_0^{-1}2^{\frac{p^+-q}{p^-}}\left(\rho_2^{\frac{1}{q}}+1\right)^{\frac{q}{p^-}}=C_0^{-1}2^{\frac{p^+-p^-}{p^-}}\left(\rho_2^{\frac{1}{p^-}}+1\right).
\end{equation}
On account of \eqref{eq2.22}, therefore, by Lemma \ref{lemma2.3} we see that
\begin{equation}\label{eq2.33}
\Vert u_0\Vert_{\infty}\le C_0^{-1}2^{\frac{p^+-p^-}{p^-}}\left(\rho_2^{\frac{1}{p^-}}+1\right),
\end{equation}
as claimed in the localisation.  Moreover, inequality \eqref{eq2.33} also implies that
\begin{equation}
u(t)\le\Vert u\Vert_{\infty}\le M_{\rho}^*\text{, }t\in[0,1]\text{, }u\in\partial\widehat{V}_{\rho},\notag
\end{equation}
which justifies the alteration in the second set in subscript of both $f^m$ and $f^M$.  And this completes the proof of the corollary.
\end{proof}

\begin{corollary}\label{corollary2.11}
Suppose both that $b\equiv\bm{1}$ and $p(t)\equiv p_0$, where $p_0>1$.  In addition, suppose that the numbers $\rho_1$ and $\rho_2$ from condition (H1.3) satisfy the following inequalities.
\vskip0.2cm
\begin{enumerate}
\item $\displaystyle\eta_0^{p_0}\left[2^{1-p_0}\left(\frac{\lambda f_{[\alpha,\beta]\times\left[\eta_0m_{\rho_1},M_{\rho_1}^*\right]}^{m}}{A\left(\rho_1\right)}\max_{\tau\in[0,1]}\int_{\alpha}^{\beta}G(\tau,s)\ ds\right)^{p_0}-1\right]>\frac{\rho_1}{\beta-\alpha}$
\vskip0.2cm
\item $\displaystyle\left(\frac{\lambda f_{[0,1]\times\left[0,M_{\rho_2}^*\right]}^{M}G^M}{A\left(\rho_2\right)}\right)^{p_0}<\rho_2$
\end{enumerate}
\vskip0.2cm
If each of conditions (H1)--(H3) holds, then problem \eqref{eq1.1}, subject to the boundary data encoded by the Green's function $G$, has at least one positive solution $u_0\in\widehat{V}_{\rho_2}\setminus\overline{\widehat{V}_{\rho_1}}$ satisfying the localisation
\begin{equation}
\rho_1^{\frac{1}{p_0}}\le\Vert u_0\Vert_{\infty}\le\frac{1}{C_0}\left(\rho_2^{\frac{1}{p_0}}+1\right).\notag
\end{equation}
\end{corollary}

\begin{proof}
Similar to the proof of Corollary \ref{corollary2.10}, the only material change to the proof of Theorem \ref{theorem2.8} is the localisation argument.  In this case, from Lemma \ref{lemma2.9} we obtain
\begin{equation}
\inf_{q\in\left(1,p^-\right)}C_0^{-1}2^{\frac{p^+-q}{p^-}}\left(\rho_2^{\frac{1}{q}}+1\right)^{\frac{q}{p^-}}=C_0^{-1}\left(\rho_2^{\frac{1}{p^-}}+1\right),\notag
\end{equation}
keeping in mind that $p^-=p_0=p^+$.  At the same time, note that
\begin{equation}
\varepsilon\left(\rho_1,\bm{1}\right)\equiv\bm{0},\notag
\end{equation}
owing to the fact that $p^-=p_0=p^+$.
\end{proof}

\begin{remark}\label{remark2.9}
Let's compare Theorem \ref{theorem2.8} to \cite[Theorem 2.12]{goodrich23}, seeing as each treats the variable exponent setting.  So, unlike in \cite{goodrich23}, we do not assume that the function $f$ satisfies any particular growth condition.  In particular, in \cite{goodrich23} it was assumed that
\begin{equation}
f(t,u)\ge c_1u^q\notag
\end{equation}
whenever $(t,u)\in[\alpha,\beta]\times[0,+\infty)$, with $c_1>0$ and $q\ge0$.  In Theorem \ref{theorem2.8} we have not made any such assumption, and this is a consequence of the fundamental inequality in Lemma \ref{lemma2.1}, an observation that was not utilised in \cite{goodrich23}.  In addition, as seen from conditions (2)--(4) of \cite[Theorem 2.12]{goodrich23}, three conditions were imposed on the nonlocal coefficient $A$, beyond the positivity condition on the interval $\big[\rho_1,\rho_2\big]$.  By contrast, in Theorem \ref{theorem2.8} we have only imposed two conditions on $A$ beyond the positivity condition.  All in all, the use of Lemma \ref{lemma2.1} allows us to streamline the analysis of the variable exponent problem in a way that was not possible in \cite{goodrich23}.
\end{remark}

We conclude with an example to illustrate the application of the results presented in this section.

\begin{example}\label{example2.12}
Since the case in which $b\equiv\bm{1}$ is an important special case, we will illustrate the application of Corollary \ref{corollary2.10}.  Define the nonlocal coefficient $A$ by
\begin{equation}
A(t):=\frac{1000}{3}t\sin{\left(\frac{\pi}{6}t\right)}.\notag
\end{equation}
If we take $G$ defined by \eqref{eq2.1mmm}, then we we are considering a nonlocal Dirichlet problem.  As in \cite{goodrich23}, let us choose 
\begin{equation}
p(t):=\frac{7}{2}+\frac{3}{2}\cos(t).\notag
\end{equation}
Therefore, we considering the following problem.
\begin{equation}\label{eq2.35}
\begin{split}
-A\left(\int_0^1\big(u(s)\big)^{\frac{7}{2}+\frac{3}{2}\cos{s}}\ ds\right)u''(t)&=\lambda f\big(t,u(t)\big)\text{, }0<t<1\\
u(0)&=0\\
u(1)&=0
\end{split}
\end{equation}

Now, for the choice of the variable exponent $p$, we see that, for all $t \in \mathbb{R}$,
\begin{equation}
2=:p^-\leq p(t)\leq p^+:=5.\notag
\end{equation}
Moreover, if we select
\begin{equation}
\alpha:=\frac{1}{4}\text{ and }\beta:=\frac{3}{4},\notag
\end{equation}
then for $G$ it is known (see \cite{erbe1}) that
\begin{equation}
\eta_0=\frac{1}{4}.\notag
\end{equation}
One can then show both that
\begin{equation}
\max_{\tau\in[0,1]}\int_{\alpha}^{\beta}G(\tau,s)\ ds=\frac{3}{32}\notag
\end{equation}
and that 
\begin{equation}
\max_{\tau\in[0,1]}\int_{0}^{1}G(\tau, s)\ ds=\frac{1}{8}.\notag
\end{equation}

Now, we choose $\rho_1$ and $\rho_2$ as follows.
\begin{equation}
\begin{split}
\rho_1&=\frac{1}{2500}\\
\rho_2&=3\notag
\end{split}
\end{equation}
Then conditions (1) and (2) of Corollary \ref{corollary2.10} are the following.
\vskip0.2cm
\begin{enumerate}
\item $\displaystyle\lambda f^m_{\left[\frac{1}{4},\frac{3}{4}\right]\times\left[\frac{1}{200},\frac{102}{25}\sqrt{2}\right]}>\left(\frac{256}{375}\sqrt{\frac{379}{3}}\right)\sin{\frac{\pi}{15000}}$
\vskip0.2cm
\item $\displaystyle\lambda f^M_{\left[0, 1\right]\times\left[0,4\sqrt2(\sqrt3+1)\right]}<8000\sqrt[5]{3}$
\end{enumerate}
\vskip0.2cm
Note that condition (H1.3) is also satisfied as $A(t)>0$ whenever $\displaystyle t\in\left[\frac{1}{2500},3\right]$.  Hence provided that $\lambda$ and $f$ satisfy (to three decimal places of accuracy) both
\begin{equation}
\lambda f^m_{\left[\frac{1}{4},\frac{3}{4}\right] \times \left[\frac{1}{200},\frac{102}{25}\sqrt{2}\right]}>\left(\frac{256}{375}\sqrt{\frac{379}{3}}\right)\sin{\frac{\pi}{15000}}\approx0.002\notag
\end{equation}
and
\begin{equation}
\lambda f^M_{\left[0, 1\right]\times\left[0, 4\sqrt2(\sqrt3+1)\right]}<8000\sqrt[5]{3}\approx9965.848,\notag
\end{equation}
then Corollary \ref{corollary2.10} implies that problem \eqref{eq2.35} has at least one positive solution, $u_0$, satisfying
\begin{equation}
u_0\in\widehat{V}_3\setminus\overline{\widehat{V}_{\frac{1}{2500}}}.\notag
\end{equation}
Moreover, we can also conclude from Corollary \ref{corollary2.10} that $u_0$ satisfies the localisation
\begin{equation}
0.02=\frac{1}{50}<\Vert u_0\Vert_{\infty}<4\sqrt{2}\big(\sqrt{3}+1\big)\approx15.455.\notag
\end{equation}
\end{example}

\begin{remark}\label{remark2.14}
We note, in Example \ref{example2.12}, both that
\begin{equation}
A(0)=0\notag
\end{equation}
and that
\begin{equation}
\liminf_{t\to+\infty}A(t)=-\infty.\notag
\end{equation}
As suggested in Section 1, we emphasise that these are unusual allowances in the theory of nonlocal differential equations.
\end{remark}

\section{Existence Theory for \eqref{eq1.1} in Case $0<p(x)\le1$}

In this section we recast the results of Section 2 in case $0<p(x)\le1$.  For the most part, there are few dissimilarities between the cases $0<p(x)\le1$ and $1<p(x)<+\infty$.  However, as will be seen in the sequel, a few do arise.  We begin by rephrasing condition (H2) from Section 2.
\begin{list}{}{\setlength{\leftmargin}{.5in}\setlength{\rightmargin}{0in}}
\item[\textbf{H2.A:}] The function $p\ : \ [0,1]\rightarrow(0,1]$ is continuous, and there exist real numbers $p^-$ and $p^+$ such that, for each $t\in[0,1]$,
\begin{equation}
0<p^-\le p(t)\le p^+\le1.\notag
\end{equation}
\end{list}
Condition (H2.A) will be in force throughout this section.

Our first couple results, namely Lemma \ref{lemma3.1} and Corollary \ref{corollary3.2}, are the analogues of Lemma \ref{lemma2.1} and Corollary \ref{corollary2.2} from Section 2.  As one sees, very few changes occur.

\begin{lemma}\label{lemma3.1}
Let $f\ : \ [0,1]\rightarrow[0,+\infty)$ be given.  Suppose that $p\ : \ [0,1]\rightarrow(0,1]$ satisfies condition (H2.A).  Then, given any constant $q$ satisfying $1\le q$, for each $t\in[0,1]$ it holds that
\begin{equation}
\big(f(t)\big)^{\frac{p(t)}{q}}\ge\big(f(t)\big)^{\frac{p^-}{q}}-1.\notag
\end{equation}
\end{lemma}

\begin{proof}
Note that $\displaystyle\frac{p(t)}{q}\le1$ for each $t\in[0,1]$.  Then, for each $0\le t\le 1$, keeping in mind that concave functions are subadditive, it follows that 
\begin{equation}\label{eq3.1}
\big(f(t)+1\big)^{\frac{p(t)}{q}}\le\big(f(t)\big)^{\frac{p(t)}{q}}+1.
\end{equation}
So, in light of \eqref{eq3.1} we deduce that
\begin{equation}\label{eq3.2}
\big(f(t)\big)^{\frac{p(t)}{q}}+1\ge\big(f(t)+1\big)^{\frac{p(t)}{q}}\ge\big(f(t)\big)^{\frac{p^-}{q}},
\end{equation}
for each $t\in[0,1]$.  Consequently, inequality \eqref{eq3.2} implies that
\begin{equation}
\big(f(t)\big)^{\frac{p(t)}{q}}\ge\big(f(t)\big)^{\frac{p^-}{q}}-1,\notag
\end{equation}
for each $t\in[0,1]$, as claimed.  And this completes the proof of the lemma.
\end{proof}

\begin{corollary}\label{corollary3.2}
Suppose that the hypotheses of Lemma \ref{lemma3.1} are true.  Then under the additional assumption that $f\in\mathscr{C}\big([0,1];[0,+\infty)\big)$ it holds that
\begin{equation}
\int_0^1\big(f(t)\big)^{\frac{p(t)}{q}}\ dt\ge\int_0^1\big(f(t)\big)^{\frac{p^-}{q}}\ dt-1.\notag
\end{equation}
\end{corollary}

We next seek to provide analogues of both Lemmata \ref{lemma2.3aaa}--\ref{lemma2.3} and Corollary \ref{corollary2.4}.  Mostly, these previous results analogise almost without alteration.  The one principal alteration is that in analogising Lemma \ref{lemma2.3} we no longer use the coercivity of the functional $u\mapsto(a*u)(1)$.  Instead we use the Harnack-like inequality $\displaystyle\min_{t\in[\alpha,\beta]}u(t)\ge\eta_0\Vert u\Vert_{\infty}$; this is similar to the approach taken earlier in \cite{goodrich11}, which treated the \emph{constant} exponent case in the concave setting.

\begin{lemma}\label{lemma3.3aaa}
Suppose that $u\in\partial\widehat{V}_{\rho}$ for some $\rho>0$.  Then
\begin{equation}
\Vert u\Vert_{\infty}\ge\left(\frac{\rho}{(b*\bm{1})(1)}\right)^{\frac{1}{p^+}}+\varepsilon(\rho,b),\notag
\end{equation}
where
\begin{equation}
\varepsilon(\rho,b):=\begin{cases} \left(\frac{\rho}{(b*\bm{1})(1)}\right)^{\frac{1}{p^-}}-\left(\frac{\rho}{(b*\bm{1})(1)}\right)^{\frac{1}{p^+}}\text{, }&0<\left(\frac{\rho}{(b*\bm{1})(1)}\right)^{\frac{1}{p^+}}<1\\ 0\text{, }&\left(\frac{\rho}{(b*\bm{1})(1)}\right)^{\frac{1}{p^+}}\ge1\end{cases}.\notag
\end{equation}
\end{lemma}

\begin{proof}
The proof is unchanged from Lemma \ref{lemma2.3aaa}.
\end{proof}

\begin{lemma}\label{lemma3.3}
Assume that, for some $q>1$,
\begin{equation}
b^{\frac{1}{1-q}}\in L^1\big((0,1]\big).\notag
\end{equation}
For each $\rho>0$, if $u\in\widehat{V}_{\rho}$, then
\begin{equation}
\Vert u\Vert_{\infty}<\frac{1}{\eta_0}(\beta-\alpha)^{-\frac{q}{p^-}}\left[\rho^{\frac{1}{q}}\left(\left(b^{\frac{1}{1-q}}*\bm{1}\right)(1)\right)^{\frac{q-1}{q}}+1\right]^{\frac{q}{p^-}}.\notag
\end{equation}
Furthermore, if $u\in\partial\widehat{V}_{\rho}$, then
\begin{equation}
\Vert u\Vert_{\infty}\le\frac{1}{\eta_0}(\beta-\alpha)^{-\frac{q}{p^-}}\left[\rho^{\frac{1}{q}}\left(\left(b^{\frac{1}{1-q}}*\bm{1}\right)(1)\right)^{\frac{q-1}{q}}+1\right]^{\frac{q}{p^-}}.\notag
\end{equation}
\end{lemma}

\begin{proof}
Keeping in mind that $u\in\widehat{V}_{\rho}$, by means of the reverse H\"{o}lder inequality we first calculate
\begin{equation}\label{eq3.3}
\begin{split}
\rho>\left(b*u^{p(\cdot)}\right)(1)&=\int_0^1b(1-s)\big(u(s)\big)^{p(s)}\ ds\\
&\ge\left(\int_0^1\big(b(1-s)\big)^{\frac{1}{1-q}}\ ds\right)^{1-q}\left(\int_0^1\big(u(s)\big)^{\frac{p(s)}{q}}\ ds\right)^{q}\\
&\ge\left(\left(b^{\frac{1}{1-q}}*\bm{1}\right)(1)\right)^{1-q}\left[-1+\int_0^1\big(u(s)\big)^{\frac{p^-}{q}}\ ds\right]^q,
\end{split}
\end{equation}
where the second inequality follows from Lemma \ref{lemma3.1}.  Now, using that $u\in\mathscr{K}$ we recall that
\begin{equation}\label{eq3.4}
\min_{t\in[\alpha,\beta]}u(t)\ge\eta_0\Vert u\Vert_{\infty}.
\end{equation}
So, \eqref{eq3.3}--\eqref{eq3.4} imply that
\begin{equation}\label{eq3.5}
\begin{split}
\rho&>\left(\left(b^{\frac{1}{1-q}}*\bm{1}\right)(1)\right)^{1-q}\left[-1+\int_0^1\big(u(s)\big)^{\frac{p^-}{q}}\ ds\right]^q\\
&\ge\left(\left(b^{\frac{1}{1-q}}*\bm{1}\right)(1)\right)^{1-q}\left[-1+\int_{\alpha}^{\beta}\big(\eta_0\Vert u\Vert_{\infty}\big)^{\frac{p^-}{q}}\ ds\right]^q\\
&=\left(\left(b^{\frac{1}{1-q}}*\bm{1}\right)(1)\right)^{1-q}\left[-1+(\beta-\alpha)\big(\eta_0\Vert u\Vert_{\infty}\big)^{\frac{p^-}{q}}\right]^q.
\end{split}
\end{equation}
Therefore, from inequality \eqref{eq3.5} we conclude that
\begin{equation}
\Vert u\Vert_{\infty}<\frac{1}{\eta_0}(\beta-\alpha)^{-\frac{q}{p^-}}\left[\rho^{\frac{1}{q}}\left(\left(b^{\frac{1}{1-q}}*\bm{1}\right)(1)\right)^{\frac{q-1}{q}}+1\right]^{\frac{q}{p^-}},\notag
\end{equation}
as claimed.

On the other hand, if $u\in\partial\widehat{V}_{\rho}$, then the only change to the preceding calculations is that in inequalities \eqref{eq3.3} and \eqref{eq3.5} the strict inequalities change to non-strict inequalities.  And this completes the proof.
\end{proof}

\begin{corollary}\label{corollary3.4}
For each $\rho>0$ the set $\widehat{V}_{\rho}$ is bounded.
\end{corollary}

\begin{proof}
This is immediate from Lemma \ref{lemma3.3}.
\end{proof}

We note that Lemma \ref{lemma2.6} holds without alteration in case $0<p(x)\le1$.  Therefore, we proceed directly to the analogy of Theorem \ref{theorem2.8}.  In both the statement and proof of Theorem \ref{theorem3.6} we denote by $\overline{M}_{\rho}$ the number
\begin{equation}
\overline{M}_{\rho}:=\frac{1}{\eta_0}(\beta-\alpha)^{-\frac{q}{p^-}}\left[\rho^{\frac{1}{q}}\left(\left(b^{\frac{1}{1-q}}*\bm{1}\right)(1)\right)^{\frac{q-1}{q}}+1\right]^{\frac{q}{p^-}},\notag
\end{equation}
for any real number $\rho>0$.  Note that the number $G^M$ continues to have the same meaning as it did in Section 2.

\begin{theorem}\label{theorem3.6}
Suppose for some $q>1$ that $b^{\frac{1}{1-q}}\in L^1\big((0,1]\big)$.  In addition, suppose that the numbers $\rho_1$ and $\rho_2$ from condition (H1.3) satisfy the following inequalities.
\vskip0.2cm
\begin{enumerate}
\item $\displaystyle\left[\left(\frac{\lambda f_{[\alpha,\beta]\times\left[\eta_0m_{\rho_1},\overline{M}_{\rho_1}\right]}^{m}}{A\left(\rho_1\right)}\max_{\tau\in[0,1]}\int_{\alpha}^{\beta}G(\tau,s)\ ds\right)^{p^-}-1\right]\int_{\alpha}^{\beta}b(1-s)\ ds>\frac{\rho_1}{\eta_0^{p^+}}$
\vskip0.2cm
\item $\displaystyle\max\left\{\left(\frac{\lambda f_{[0,1]\times\left[0,\overline{M}_{\rho_2}\right]}^{M}G^M}{A\left(\rho_2\right)}\right)^{p^+},\left(\frac{\lambda f_{[0,1]\times\left[0,\overline{M}_{\rho_2}\right]}^{M}G^M}{A\left(\rho_2\right)}\right)^{p^-}\right\}<\frac{\rho_2}{(b*\bm{1})(1)}$
\end{enumerate}
\vskip0.2cm
If each of conditions (H1), (H2.A), and (H3) holds, then problem \eqref{eq1.1}, subject to the boundary data encoded by the Green's function $G$, has at least one positive solution $u_0\in\widehat{V}_{\rho_2}\setminus\overline{\widehat{V}_{\rho_1}}$ satisfying the localisation
\begin{equation}
\left(\frac{\rho_1}{(b*\bm{1})(1)}\right)^{\frac{1}{p^+}}+\varepsilon\left(\rho_1,b\right)\le\Vert u\Vert_{\infty}\le\frac{1}{\eta_0}(\beta-\alpha)^{-\frac{q}{p^-}}\left[\rho_2^{\frac{1}{q}}\left(\left(b^{\frac{1}{1-q}}*\bm{1}\right)(1)\right)^{\frac{q-1}{q}}+1\right]^{\frac{q}{p^-}}.\notag
\end{equation}
\end{theorem}

\begin{proof}
The only alteration to the proof of Theorem \ref{theorem2.8} is the part in which we appeal to part (1) of Lemma \ref{lemma2.4aaa}.  In particular, note that by Lemma \ref{lemma3.1} (with the selection $q=1$ in the statement of the lemma) we obtain that
\begin{equation}\label{eq3.6}
\begin{split}
\rho_1&=\left(b*u^{p(\cdot)}\right)(1)\\
&=\left(b*(Tu+\mu\bm{1})^{p(\cdot)}\right)(1)\\
&\ge\left(b*Tu^{p(\cdot)}\right)(1)\\
&\ge\int_{\alpha}^{\beta}b(1-t)\left(\eta_0\Vert Tu\Vert_{\infty}\right)^{p(t)}\ dt\\
&\ge\eta_0^{p^+}\int_{\alpha}^{\beta}b(1-t)\Vert Tu\Vert_{\infty}^{p(t)}\ dt\\
&\ge\eta_0^{p^+}\int_{\alpha}^{\beta}b(1-t)\left(\Vert Tu\Vert_{\infty}^{p^-}-1\right)\ dt\\
&=\eta_0^{p^+}\left[\Vert Tu\Vert_{\infty}^{p^-}-1\right]\int_{\alpha}^{\beta}b(1-t)\ dt\\
&\ge\eta_0^{p^+}\left[\left(\frac{\lambda f_{[\alpha,\beta]\times\left[\eta_0m_{\rho_1},\overline{M}_{\rho_1}\right]}^{m}}{A\left(\rho_1\right)}\max_{\tau\in[0,1]}\int_{\alpha}^{\beta}G(\tau,s)\ ds\right)^{p^-}-1\right]\int_{\alpha}^{\beta}b(1-t)\ dt.
\end{split}
\end{equation}
But then inequality \eqref{eq3.6}, together with condition (1) in the statement of the theorem, implies a contradiction.  Since the second part of the proof is unchanged from the second part of the proof of Theorem \ref{theorem2.8}, this completes the proof of the theorem.
\end{proof}

\begin{remark}\label{remark3.7}
We emphasise that the results of this section demonstrate the broadness of the techniques introduced in this paper.  In the constant exponent case, i.e., $p(x)\equiv p$, it is very unusual to treat simultaneously the cases $0<p\le1$ and $p>1$, much less in the variable exponent case.  Thus, this demonstrates the power of using the fundamental inequalities embodied by Lemmata \ref{lemma2.1} and \ref{lemma3.1}, together with the specialised cone $\mathscr{K}$.
\end{remark}

\section{Existence Theory for \eqref{eq1.1} in Case $0<p^-< 1 < p^+$}

As mentioned in Section 1, we can use our results from Sections 2 and 3 to accomodate the situation in which $p^-<1<p^+$.  For this reason, we will use the following condition in this section.
\begin{list}{}{\setlength{\leftmargin}{.5in}\setlength{\rightmargin}{0in}}
\item[\textbf{H2.B:}] The function $p\ : \ [0,1]\rightarrow(0,+\infty)$ is continuous, and there exist real numbers $p^-$ and $p^+$, which satisfy the inequality
\begin{equation}
0<p^-<1<p^+<+\infty,\notag
\end{equation}
such that, for each $t\in[0,1]$,
\begin{equation}
p^-\le p(t)\le p^+.\notag
\end{equation}
\end{list}
When this regime occurs, the variable growth can freely switch from a concave-type regime to a convex-type regime and vice versa, thereby permitting a very general theoretical framework.  In this brief section, we show how the the results of Sections 2 and 3 facilitate this unification.

\begin{lemma}\label{lemma4.1}
Let $f\ : \ [0,1]\rightarrow[0,+\infty)$ be given.  Suppose that $p\ : \ [0,1]\rightarrow[0,+\infty)$ satisfies condition (H2.B).  Then, given any $q\in\big[1,p^+\big)$, for each $t\in[0,1]$ it holds that
\begin{equation}
\big(f(t)\big)^{\frac{p^-}{q}}-1\le\big(f(t)\big)^{\frac{p(t)}{q}}<2^{\frac{p^+}{q}-1}\left(\big(f(t)\big)^{\frac{p^+}{q}}+1\right).\notag
\end{equation}
\end{lemma}

\begin{proof}
First note that, for each $t\in[0,1]$,
\begin{equation}\label{eq4.1}
\big(f(t)\big)^{\frac{p(t)}{q}}<\big(f(t)+1\big)^{\frac{p(t)}{q}}\le\big(f(t)+1\big)^{\frac{p^+}{q}}\le2^{\frac{p^+}{q}-1}\left(\big(f(t)\big)^{\frac{p^+}{q}}+1\right).
\end{equation}
Thus, inequality \eqref{eq4.1} establishes the desired upper bound.

On the other hand, suppose that $f(t)\ge1$ for some fixed $t\in[0,1]$.  Then
\begin{equation}\label{eq4.2}
\big(f(t)\big)^{\frac{p(t)}{q}}\ge\big(f(t)\big)^{\frac{p^-}{q}}>\big(f(t)\big)^{\frac{p^-}{q}}-1.
\end{equation}
If instead $f(t)<1$ for some fixed $t\in[0,1]$, then since
\begin{equation}
\big(f(t)\big)^{\frac{p^-}{q}}-1<0,\notag
\end{equation}
it follows trivially that
\begin{equation}\label{eq4.3}
\big(f(t)\big)^{\frac{p(t)}{q}}>\big(f(t)\big)^{\frac{p^-}{q}}-1.
\end{equation}
Combining inequalities \eqref{eq4.2}--\eqref{eq4.3} we arrive at the desired lower bound.  And this completes the proof.
\end{proof}

\begin{lemma}\label{lemma4.2}
Suppose that $u\in\partial\widehat{V}_{\rho}$ for some $\rho>0$.  Then
\begin{equation}
\Vert u\Vert_{\infty}\ge\left(\frac{\rho}{(b*\bm{1})(1)}\right)^{\frac{1}{p^+}}+\varepsilon(\rho,b),\notag
\end{equation}
where
\begin{equation}
\varepsilon(\rho,b):=\begin{cases} \left(\frac{\rho}{(b*\bm{1})(1)}\right)^{\frac{1}{p^-}}-\left(\frac{\rho}{(b*\bm{1})(1)}\right)^{\frac{1}{p^+}}\text{, }&0<\left(\frac{\rho}{(b*\bm{1})(1)}\right)^{\frac{1}{p^+}}<1\\ 0\text{, }&\left(\frac{\rho}{(b*\bm{1})(1)}\right)^{\frac{1}{p^+}}\ge1\end{cases}.\notag
\end{equation}
\end{lemma}

\begin{proof}
The proof is the same as that of Lemma \ref{lemma2.3aaa}.
\end{proof}

\begin{lemma}\label{lemma4.3}
Assume that, for some $q\in\big(1,p^+\big)$,
\begin{equation}
b^{\frac{1}{1-q}}\in L^1\big((0,1]\big).\notag
\end{equation}
For each $\rho>0$, if $u\in\widehat{V}_{\rho}$, then
\begin{equation}
\Vert u\Vert_{\infty}<\frac{1}{\eta_0}(\beta-\alpha)^{-\frac{q}{p^-}}\left[\rho^{\frac{1}{q}}\left(\left(b^{\frac{1}{1-q}}*\bm{1}\right)(1)\right)^{\frac{q-1}{q}}+1\right]^{\frac{q}{p^-}}.\notag
\end{equation}
Furthermore, if $u\in\partial\widehat{V}_{\rho}$, then
\begin{equation}
\Vert u\Vert_{\infty}\le\frac{1}{\eta_0}(\beta-\alpha)^{-\frac{q}{p^-}}\left[\rho^{\frac{1}{q}}\left(\left(b^{\frac{1}{1-q}}*\bm{1}\right)(1)\right)^{\frac{q-1}{q}}+1\right]^{\frac{q}{p^-}}.\notag
\end{equation}
\end{lemma}

\begin{proof}
The proof is the same that of Lemma \ref{lemma3.3} save but for the fact that in order to use now Lemma \ref{lemma4.1} we must have $1\le q<p^+$.
\end{proof}

With Lemmata \ref{lemma4.1}--\ref{lemma4.3} in hand, we present our existence result for this section.

\begin{theorem}\label{theorem4.4}
Suppose for some $q>1$ that $b^{\frac{1}{1-q}}\in L^1\big((0,1]\big)$.  In addition, suppose that the numbers $\rho_1$ and $\rho_2$ from condition (H1.3) are selected such that each of the following inequalities is true.
\vskip0.2cm
\begin{enumerate}
\item $\displaystyle\eta_0^{p^+}\left(\frac{\lambda f_{[\alpha,\beta]\times\left[\eta_0m_{\rho_1},\overline{M}_{\rho_1}\right]}^{m}}{A\left(\rho_1\right)}\max_{\tau\in[0,1]}\int_{\alpha}^{\beta}G(\tau,s)\ ds\right)^{p^-}\int_{\alpha}^{\beta}b(1-s)\ ds>\rho_1$
\vskip0.2cm
\item $\displaystyle\max\left\{\left(\frac{\lambda f_{[0,1]\times\left[0,\overline{M}_{\rho_2}\right]}^{M}G^M}{A\left(\rho_2\right)}\right)^{p^+},\left(\frac{\lambda f_{[0,1]\times\left[0,\overline{M}_{\rho_2}\right]}^{M}G^M}{A\left(\rho_2\right)}\right)^{p^-}\right\}<\frac{\rho_2}{(b*\bm{1})(1)}$
\end{enumerate}
\vskip0.2cm
If each of conditions (H1), (H2.B), and (H3) holds, then problem \eqref{eq1.1}, subject to the boundary data encoded by the Green's function $G$, has at least one positive solution $u_0\in\widehat{V}_{\rho_2}\setminus\overline{\widehat{V}_{\rho_1}}$ satisfying the localisation
\begin{equation}
\left(\frac{\rho_1}{(b*\bm{1})(1)}\right)^{\frac{1}{p^+}}+\varepsilon\left(\rho_1,b\right)\le\Vert u\Vert_{\infty}\le\frac{1}{\eta_0}(\beta-\alpha)^{-\frac{q}{p^-}}\left[\rho_2^{\frac{1}{q}}\left(\left(b^{\frac{1}{1-q}}*\bm{1}\right)(1)\right)^{\frac{q-1}{q}}+1\right]^{\frac{q}{p^-}}.\notag
\end{equation}
\end{theorem}

\begin{proof}
One obtains \eqref{eq2.20} since $p^+>1$.  And, likewise, one obtains \eqref{eq3.6} since $p^-<1$.  So, combining the proofs of Theorems \ref{theorem2.8} and \ref{theorem3.6} yields the proof of the present theorem.
\end{proof}

\begin{remark}\label{remark4.5}
Note that Theorem \ref{theorem4.4} is, in fact, identical to Theorem \ref{theorem3.6}.  Thus, in both the concave-type case (i.e., Section 3) and the mixed case (i.e., Section 4) we obtain the same existence result.  In the convex-type case (i.e., Section 2) the existence result differs from Theorem \ref{theorem4.4}.  Also note that none of the cases requires the extra growth condition on $f$ that was imposed in \cite{goodrich23}.  Moreover, neither \cite{goodrich23} nor any work of which we are aware treats the cases identified in Sections 3 and 4.
\end{remark}


\end{document}